\documentclass{article}

\usepackage{indentfirst}
\usepackage[english]{babel}
\usepackage{amsmath, amssymb}
\usepackage{amsthm}
\usepackage{amsfonts}
\usepackage{enumerate}
\usepackage[latin1]{inputenc}
\usepackage{graphics}
\usepackage{comment}

\newtheorem{thm}{Theorem}
\newtheorem{cor}[thm]{Corollary}

\newtheorem{defin}[thm]{Definition}

\begin{document}

\title{Shape differentiability of the eigenvalues of elliptic systems}
\author{Davide Buoso\footnote{University of Padova; email: dbuoso@math.unipd.it.}}
\date{ }
\maketitle

Let $\Omega$ be a bounded open set in $\mathbb{R}^N$ of class $C^1$, $m\in\mathbb{N}$. By $H^1(\Omega)$ we denote the Sobolev space of functions in $L^2(\Omega)$ with derivatives in $L^2(\Omega)$, and by $H^1_0(\Omega)$ we denote the closure in $H^1(\Omega)$ of the space of $C^{\infty}$-functions with compact support in $\Omega$.

We consider the following eigenvalue problem in the weak form
\begin{equation}
\label{problema}
\int_{\Omega}\sum_{\alpha,\beta=1}^N\sum_{i,j=1}^ma_{\alpha\beta}^{ij}\frac{\partial u_i}{\partial x_{\alpha}}\frac{\partial\varphi_j}{\partial x_{\beta}}dx=\lambda\int_{\Omega}u\cdot\varphi dx,
\end{equation}
for any $\varphi\in V(\Omega)^m$, in the unknowns $u\in V(\Omega)^m$ (the eigenfunction), $\lambda\in\mathbb{R}$ (the eigenvalue), where $V(\Omega)$ denotes either $H^1_0(\Omega)$ (for Dirichlet boundary conditions) or $H^1(\Omega)$ (for Neumann boundary conditions).

Note that the classical formulation of the Dirichlet problem reads
	\begin{equation}
	\label{dirichletsys}
	\left\{\begin{array}{ll}
-\sum_{\alpha,\beta=1}^N\sum_{i=1}^ma_{\alpha\beta}^{ij}\frac{\partial^2 u_i}{\partial x_{\alpha}\partial x_{\beta}}=\lambda u_{j},j=1,\dots,m,&\mathrm{in\ }\Omega,\\
	u=0,&\mathrm{on\ }\partial\Omega,
	\end{array}\right.
	\end{equation}
	while the classical formulation of the Neumann problem reads
	\begin{equation}
\label{neumannsys}
	\left\{\begin{array}{ll}
-\sum_{\alpha,\beta=1}^N\sum_{i=1}^ma_{\alpha\beta}^{ij}\frac{\partial^2 u_i}{\partial x_{\alpha}\partial x_{\beta}}=\lambda u_{j},j=1,\dots,m,&\mathrm{in\ }\Omega,\\
	\sum_{\alpha,\beta=1}^N\sum_{i=1}^ma_{\alpha\beta}^{ij}\nu_{\beta}\frac{\partial u_i}{\partial x_{\alpha}}=0,j=1,\dots,m,&\mathrm{on\ }\partial\Omega,
	\end{array}\right.
\end{equation}
where $\nu$ denotes the outer unit normal to $\partial\Omega$.

Here and in the sequel $a_{\alpha\beta}^{ij}\in\mathbb{R}$ are constant coefficients satisfying $a_{\alpha\beta}^{ij}=a_{\beta\alpha}^{ji}$ and the Legendre-Hadamard condition, i.e., 
\begin{equation}
\label{lhc}
\sum_{\alpha,\beta=1}^N\sum_{i,j=1}^ma_{\alpha\beta}^{ij}\xi_i\xi_j\eta_{\alpha}\eta_{\beta}\geq\theta|\xi|^2|\eta|^2,\ \ \forall\xi\in\mathbb{R}^m,\forall\eta\in\mathbb{R}^N,
\end{equation}
for some $\theta>0$.

We consider in $H^1(\Omega)^m$ the bilinear form
\begin{equation}\label{systemform1}
<u,v>=\int_{\Omega}\sum_{\alpha,\beta=1}^N\sum_{i,j=1}^ma_{\alpha\beta}^{ij}\frac{\partial u_i}{\partial x_{\alpha}}\frac{\partial v_j}{\partial x_{\beta}}dx,
\end{equation}
 for any $u,v\in H^1(\Omega)^m$.  Note that, for instance, it is possible to prove that the bilinear form (\ref{systemform1}) defines on $H^1_0(\Omega)^m$ a scalar product whose induced norm is equivalent to the standard one. 

Note that problem (\ref{problema}) includes some important problems in linear elasticity. For instance, the choice $a_{\alpha\beta}^{ij}=\delta_{ij}\delta_{\alpha\beta}+k\delta_{i\alpha}\delta_{j\beta}$, where $\delta_{ij}$ is the Kronecher delta and $k\ge0$ a constant, leads to the Lam\'e eigenvalue problem
\begin{equation}\label{lame}
\left\{\begin{array}{ll}
-\Delta u -k\nabla{\rm div}u=\lambda u,&\mathrm{in\ }\Omega,\\
	u=0,&\mathrm{on\ }\partial\Omega.
	\end{array}\right.
\end{equation}
Problem (\ref{lame}) is very similar to the Reissner-Mindlin system
\begin{equation}
	\label{class1}
	\left\{
	\begin{array}{l l}
	-\frac{\mu}{12}\Delta\beta-\frac{\mu+\lambda}{12}\nabla\mathrm{div}\beta  -\frac{\mu k}{t^2}(\nabla w-\beta)
		= \frac{\gamma t^2}{12}\beta, & \mathrm{in}\ \Omega , \\
	-\frac{\mu k}{t^2}(\Delta w -\mathrm{div}\beta)=\gamma w, & \mathrm{in}\ \Omega,\\
	\beta = 0, \ \ w=0, & \mathrm{on}\ \partial\Omega,
	\end{array}
	\right.
	\end{equation}
which arises in the study of the vibrations of a clamped plate. Here $\mu,\lambda,k$ and $t$ are physical constants, $\gamma$ is the eigenvalue and $(\beta,w)$ is the eigenvector. Note that the current discussion does not comprehend problem (\ref{class1}), since it presents lower order terms. However, the arguments we use can be easily adapted in order to treat problem (\ref{class1}) as well (see \cite{bulareis}).

Thanks to condition (\ref{lhc}), it is possible to show that the eigenvalues of problem (\ref{problema}) are non-negative, have finite multiplicity and can be represented as a non-decreasing divergent sequence $\lambda_k[\Omega]$, $k\in\mathbb{N}$ where each eigenvalue is repeated according to its multiplicity. In particular,
\begin{equation*}
\lambda_k[\Omega]=\min_{\substack{E\subset V(\Omega )^m \\ {\rm dim} E=k}}\max_{\substack{u\in E \\ u\neq 0}}R[u],
\end{equation*}
for all $k\in\mathbb{N}$, where $V(\Omega)$ denotes either $H^1_0(\Omega)$ (for problem (\ref{dirichletsys})) or $H^1(\Omega)$ (for problem (\ref{neumannsys})), and $R[u]$ is the Rayleigh quotient defined by
\begin{equation*}
R[u]=\frac{\int_{\Omega}\sum_{\alpha,\beta=1}^N\sum_{i,j=1}^ma_{\alpha\beta}^{ij}\frac{\partial u_i}{\partial x_{\alpha}}\frac{\partial u_j}{\partial x_{\beta}}dx}{\int_{\Omega}|u|^2dx}.
\end{equation*}

In order to shorten our notation, in the sequel we shall use Einstein notation, hence summation symbols will be dropped.

In Section \ref{analitico} we examine the problem of shape differentiability of the eigenvalues of problem (\ref{problema}). We consider problem (\ref{problema}) in $\phi(\Omega)$ and pull it back to $\Omega$, where $\phi$ belongs to a suitable class of diffeomorphisms. This analysis was exploited in \cite{lala2004, lala2007} for the Laplace operator, in \cite{bula2013, buosohinged, bupro} for polyharmonic operators and in \cite{bulareis} for the Reissner-Mindlin system (\ref{class1}). In particular, we derive Hadamard-type formulas for the symmetric functions of the eigenvalues of problem (\ref{problema}).

In Section \ref{critico} we consider the problem of finding critical points for the symmetric functions of the eigenvalues of problem (\ref{problema}), under volume constraint. This is strictly related to the problem of shape optimization of the eigenvalue (see \cite{henrot} for a detailed discussion on the topic). Similarly to what was done in \cite{bula2013,buosohinged,bulareis,bupro,lalacri}, we provide a characterization for the critical domains, and show that, for a particular class of coefficients $a_{\alpha\beta}^{ij}$, balls are critical domains for all the symmetric functions of the eigenvalues.

%%%%%%%%%%%%%55%%%%%%%%%%%%%%%%%%%%%%%%%%%%%%%%%%%%%%%%%%%%%%%%%%%%%%%%%%%%%%%5
%%%%%%%%%%%%%55%%%%%%%%%%%%%%%%%%%%%%%%%%%%%%%%%%%%%%%%%%%%%%%%%%%%%%%%%%%%%%%5
%%%%%%%%%%%%%5                                 ANALITICITY
%%%%%%%%%%%%%55%%%%%%%%%%%%%%%%%%%%%%%%%%%%%%%%%%%%%%%%%%%%%%%%%%%%%%%%%%%%%%%5
%%%%%%%%%%%%%55%%%%%%%%%%%%%%%%%%%%%%%%%%%%%%%%%%%%%%%%%%%%%%%%%%%%%%%%%%%%%%%5

\section{Analyticity results}
\label{analitico}

Let $\Omega $ be a bounded open set in ${\mathbb{R}}^N$ of class $C^1$.  We shall consider problem (\ref{problema}) in a family of open sets  parameterized by suitable diffeomorphisms $\phi $
defined on $\Omega $. Namely, we set
$$
{\mathcal{A}}_{\Omega }=\biggl\{\phi\in C^1(\overline\Omega\, ; {\mathbb{R}}^N ):\ \inf_{\substack{x_1,x_2\in \overline\Omega \\ x_1\ne x_2}}\frac{|\phi(x_1)-\phi(x_2)|}{|x_1-x_2|}>0 \biggr\},
$$
where $C^1(\overline\Omega\, ; {\mathbb{R}}^N )$ denotes the space of all functions from $\overline\Omega $ to ${\mathbb{R}}^N$ of class $C^1$.  Note that if $\phi \in {\mathcal{A}}_{\Omega }$ then $\phi $ is injective, Lipschitz continuous and $\inf_{\overline\Omega }|{\rm det }\nabla \phi |>0$. Moreover, $\phi (\Omega )$ is a bounded open set of class $C^1$ and the inverse map $\phi^{(-1)}$ belongs to  ${\mathcal{A}}_{\phi(\Omega )}$.  Thus it is natural to consider problem (\ref{problema}) on $\phi (\Omega )$ and study  the dependence of $\lambda_k[\phi (\Omega )]$ on $\phi \in {\mathcal{A}}_{\Omega }$. To do so, we endow the space $C^1(\overline\Omega\, ; {\mathbb{R}}^N )$ with its usual norm.
%$\| f \|_{C^n_b(\Omega\, ;{\mathbb{R}}^N )}=\sup_{|\alpha|\le n,\ x\in\Omega  } |D^{\alpha }f(x)|$. 
Note that ${\mathcal{A}}_{\Omega }$ is an open set in  $C^1(\overline\Omega\, ;{\mathbb{R}}^N )$, see \cite[Lemma~3.11]{lala2004}.
    Thus, it makes sense to  study differentiability and analyticity properties of the maps $\phi \mapsto \lambda_k[\phi (\Omega )]$ defined for $\phi \in {\mathcal{A}}_{\Omega }$.
   For simplicity, we write  $\lambda_k[\phi ]$ instead of $\lambda_k[\phi (\Omega )]$.
   We fix a finite set of indexes $F\subset \mathbb{N}$
and we consider those maps $\phi\in {\mathcal{A}}_{\Omega }$ for which the eigenvalues
with indexes in $F$  do not coincide with eigenvalues with indexes not
in $F$; namely we set
$$
{\mathcal { A}}_{F, \Omega }= \left\{\phi \in {\mathcal { A}}_{\Omega }:\
\lambda_k[\phi ]\ne \lambda_l[\phi],\ \forall\  k\in F,\,   l\in \mathbb{N}\setminus F
\right\}.
$$
It is also convenient to consider those maps $\phi \in {\mathcal { A}}_{F, \Omega } $ such that all the eigenvalues with index in $F$
 coincide and set
$$
\Theta_{F, \Omega } = \left\{\phi\in {\mathcal { A}}_{F, \Omega }:\ \lambda_{k_1}[\phi ]=\lambda_{k_2}[\phi ],\, \
\forall\ k_1,k_2\in F  \right\} .
$$

 For $\phi \in {\mathcal { A}}_{F, \Omega }$, the elementary symmetric functions of the eigenvalues with index in $F$ are defined by
\begin{equation}
\label{sym1}
\Lambda_{F,s}[\phi ]=\sum_{ \substack{ k_1,\dots ,k_s\in F\\ k_1<\dots <k_s} }
\lambda_{k_1}[\phi ]\cdots \lambda_{k_s}[\phi ],\ \ \ s=1,\dots , |F|.
\end{equation}

We have the following 

\begin{thm}
\label{duesettesys}
Let $\Omega $ be a bounded open set in ${\mathbb{R}}^N$ of class $C^1$ and  $F$ be a finite set in  ${\mathbb{N}}$. 
The set ${\mathcal { A}}_{F, \Omega }$ is open in
$\mathcal{A}_{\Omega}$, and the real-valued maps %which take $\phi\in {\mathcal { A}}_{F, \Omega } $ to 
$ \Lambda_{F,s}$ are real-analytic on  ${\mathcal { A}}_{F, \Omega }$, for all $s=1,\dots , |F|$. 
Moreover, if $\tilde \phi\in \Theta_{F, \Omega }  $ is such that the eigenvalues $\lambda_k[\tilde \phi]$ assume the common value $\lambda_F[\tilde \phi ]$ for all $k\in F$, and  $\tilde \phi (\Omega )$ is of class $C^{2}$ then  the Frech\'{e}t differential of the map $\Lambda_{F,s}$ at the point $\tilde\phi $ is delivered by the formula
\begin{equation}
		\label{derivdsys}
		d|_{\phi=\tilde{\phi}}(\Lambda_{F,s})[\psi]
			=	-\lambda_F^s[\tilde{\phi}]\binom{|F|-1}{s-1}
			\sum_{l=1}^{|F|}
			\int_{\partial\tilde{\phi}(\Omega)}a_{\alpha\beta}^{ij}\frac{\partial v^{(l)}_i}{\partial y_{\alpha}}\frac{\partial v^{(l)}_j}{\partial y_{\beta}}\zeta\cdot\nu d\sigma,
			\end{equation}
			for problem (\ref{dirichletsys}), or
			\begin{multline}
		\label{derivnsys}
			d|_{\phi=\tilde{\phi}}(\Lambda_{F,s})[\psi]
			=-\lambda_F^s[\tilde{\phi}]\binom{|F|-1}{s-1}\sum_{l=1}^{|F|}\\
\int_{\partial\tilde{\phi}(\Omega)}\left(\lambda_F|v^{(l)}|^2-a_{\alpha\beta}^{ij}\frac{\partial v^{(l)}_i}{\partial y_{\alpha}}\frac{\partial v^{(l)}_j}{\partial y_{\beta}}\right)\zeta\cdot\nu d\sigma,
		\end{multline}
	for problem (\ref{neumannsys}), for all $\psi\in C^1(\overline\Omega;\mathbb{R}^N)$, where $\zeta=\psi\circ\tilde{\phi}^{(-1)}$ and $\{v^{(l)}\}_{l\in F}$ is an orthonormal basis in $V(\tilde \phi (\Omega ))^m$ (with respect to the scalar product (\ref{systemform1})) of the eigenspace associated with $\lambda_F[\tilde \phi]$.
\end{thm}

The proof of Theorem \ref{duesettesys} can be easily done adapting that of \cite[Theorem 3.38]{lala2004} for the Dirichlet problem (\ref{dirichletsys}), and that of \cite[Theorem 2.5]{lala2007} for the Neumann problem (\ref{neumannsys}), and it can be found in \cite{buosotesi}.

	\section{Isovolumetric perturbations}
\label{critico}

We consider the following extremum problems for the symmetric functions of the eigenvalues
\begin{equation}\label{min}
\min_{V[\phi ]={\rm const}}\Lambda_{F,s}[\phi ]\ \ \ {\rm or}\ \ \ \max_{   V[\phi ]={\rm const}}\Lambda_{F,s}[\phi],
\end{equation}
where $V[\phi ]$ denotes the $N$-dimensional Lebesgue measure of $\phi (\Omega )$.
Note that if $\tilde \phi \in {\mathcal{A}}_{\Omega }$ is a minimizer or maximizer in (\ref{min}) then $\tilde \phi $ is a critical domain transformation for the map $\phi\mapsto  \Lambda_{F,s}[\phi ]$ subject to volume constraint, i.e.,
\begin{equation*}%\label{inc}
{\rm Ker\ d}|_{\phi =\tilde \phi }V \subset {\rm Ker\ d}|_{\phi =\tilde \phi } \Lambda_{F,s},
\end{equation*}
where $V$ is the real valued function defined on ${\mathcal{A}}_{\Omega }$ which takes $\phi \in {\mathcal{A}}_{\Omega }$ to $V[\phi ]$.

The following theorem provides a characterization of all critical domain transformations $\phi$ (see also \cite{bula2013, buosohinged, bulareis, bupro, lalacri}).

\begin{thm}\label{moltiplicatorisys}
Let $\Omega $ be a bounded open set in ${\mathbb{R}}^N$ of class $C^1$,  and $F$ be a finite subset of ${\mathbb{N}}$.
Assume that $\tilde \phi\in \Theta_{F, \Omega }$ is such that $\tilde \phi (\Omega )$ is of class $C^{2}$ and that the eigenvalues $\lambda_j[\tilde \phi ]$ have the common value $\lambda_F[\tilde \phi ]$
for all $j\in F$.  Let $\{  v^{(l)}\}_{l\in F}$ be an orthornormal basis in $V(\tilde \phi (\Omega ))^m$ (with respect to the scalar product (\ref{systemform1})) of the eigenspace corresponding to $\lambda_F[\tilde \phi ]$. Then  $\tilde \phi$
is a critical domain transformation  for any of the functions $\Lambda_{F,s}$, $s=1,\dots , |F|$,  with  volume constraint
if and only if  there exists $c\in {\mathbb{R}}$ such that
	\begin{equation}
	\label{lacondizionedsys}
	\sum_{l=1}^{|F|}a_{\alpha\beta}^{ij}\frac{\partial v^{(l)}_i}{\partial y_{\alpha}}\frac{\partial v^{(l)}_j}{\partial y_{\beta}}=c,\ {\rm\ on\ }\partial\tilde\phi(\Omega),
	\end{equation}
for problem (\ref{dirichletsys}), or
\begin{equation}
	\label{lacondizionensys}
	\sum_{l=1}^{|F|}\left(\lambda_F|v^{(l)}|^2-a_{\alpha\beta}^{ij}\frac{\partial v^{(l)}_i}{\partial y_{\alpha}}\frac{\partial v^{(l)}_j}{\partial y_{\beta}}\right)=c,\ {\rm\ on\ }\partial\tilde\phi(\Omega),
	\end{equation}
for problem (\ref{neumannsys}).
\proof
The proof is a straightforward application of Lagrange Multipliers Theorem combined with formulas (\ref{derivdsys}) and (\ref{derivnsys}).
\endproof
\end{thm}
	
	Now we introduce the following
	
	\begin{defin}
	The operator $\mathcal{L}$ defined by
	$$\mathcal{L}(u)_{j}=-a_{\alpha\beta}^{ij}\frac{\partial^2u_i}{\partial x_{\alpha}\partial x_{\beta}}$$
	is said to be rotation invariant if there exists a group homomorphism
	\begin{equation*}
	S: O_N(\mathbb{R})\rightarrow O_m(\mathbb{R}),
	\end{equation*}
(i.e., $S(AB)=S(A)S(B)$ for all $A,B \in O_N(\mathbb{R})$) such that
	\begin{equation*}
	\mathcal{L}\left(S(R)^tu\circ R\right)=S(R)^t\mathcal{L}(u)\circ R,
	\end{equation*}
	for any $R\in O_N(\mathbb{R})$, and for any $u\in H^2_{loc}(\mathbb{R}^N)^{m}$.
	\label{rotinv}
	\end{defin}
	
	So far we have the following
		
	\begin{thm}
	\label{lepallesys}
Suppose that the operator associated with problem (\ref{problema})	is rotation invariant. 		Let $B$ be the unit ball in $\mathbb{R}^N$ centered at zero,
		and let $\lambda$ be an eigenvalue of problem (\ref{problema}) in $B$. Let $F$ be the subset of
		$\mathbb{N}\setminus\{0\}$ of all $k$ such that the $k$-th eigenvalue of problem (\ref{problema}) in $B$ coincides with $\lambda$. Let $v^{(1)},\dots,v^{(|F|)}$
		be an orthonormal basis of the eigenspace associated with the eigenvalue $\lambda$ in $V(B)^m$. Then there exists $c\in\mathbb{R}$ such that condition (\ref{lacondizionedsys})  (condition (\ref{lacondizionensys}) respectively) holds.
	\end{thm}

	\begin{proof} First of all, note that by standard regularity theory (cf.\ \cite[\S 10.3]{adn}), the functions $v^{(l)}\in C^{\infty}(\overline B)$ for all $l\in F$. Moreover, since $\{v^{(l)}\}_{l\in F}$ is an orthonormal basis with respect to the scalar product (\ref{systemform1}) and $v^{(l)}$ satisfies equation (\ref{problema}) for any $l\in F$, then $\{v^{(l)}\}_{l\in F}$ is an orthonormal basis also with respect to the standard scalar product of $L^2$.
	
		Thanks to the rotation invariance, $\{(S(R)^tv_l)\circ R:l=1,\dots, |F|\}$ is another orthonormal basis
		for the eigenspace associated with $\lambda$, whenever $R\in O_n(\mathbb{R})$, where $S(R)$ is defined as in Definition \ref{rotinv}. 
		Since both $\{v^{(l)}:l=1,\dots, |F|\}$ and $\{(S(R)^tv^{(l)})\circ R:l=1,\dots, |F|\}$
		are orthonormal bases, then there exists $A[R]\in O_N(\mathbb{R})$ with matrix $(A_{rh}[R])_{r,h=1,\dots,|F|}$ such that
		$(S(R)^tv^{(r)})\circ R=\sum_{l=1}^{|F|}A_{rl}[R]v^{(l)}$.
		This implies that
		\begin{equation}
		\sum_{r=1}^{|F|}|v^{(r)}|^2\circ R=\sum_{l=1}^{|F|}|v^{(l)}|^2,
		\end{equation}
		by the $L^2$-orthonormality, and that
		\begin{equation}
		\sum_{r=1}^{|F|}\left(a_{\alpha\beta}^{ij}\frac{\partial v^{(r)}_i}{\partial y_{\alpha}}\frac{\partial v^{(r)}_j}{\partial y_{\beta}}\right)\circ R=
		\sum_{l=1}^{|F|}\left(a_{\alpha\beta}^{ij}\frac{\partial v^{(l)}_i}{\partial y_{\alpha}}\frac{\partial v^{(l)}_j}{\partial y_{\beta}}\right),
		\end{equation}
		by the orthonormality with respect to the scalar product (\ref{systemform1}).
		This concludes the proof.
	\end{proof}
	
	Thus we get the following
	\begin{cor}
	Let $\Omega$ be a domain in $\mathbb{R}^N$ of calss $C^1$. Suppose that the operator associated with problem (\ref{problema})	is rotation invariant. Let $\tilde{\phi}\in\mathcal{A}_{\Omega}$ be such that
	$\tilde{\phi}(\Omega)$ is a ball. Let $\tilde{\lambda}$ be an eigenvalue of problem (\ref{dirichletsys}) in $\tilde{\phi}(\Omega)$,
	and let $F$ be the set of $j\in\mathbb{N}\setminus\{0\}$ such that $\lambda_j[\tilde{\phi}]=\tilde{\lambda}$.
	Then $\tilde{\phi}$ is a critical point $\Lambda_{F,s}$  under volume constraint,
	for all $s=1,\dots,|F|$.
	\end{cor}
	
	{\bf Acknowledgements.} The author is very thankful to Prof.\ Pier Domenico Lamberti for his useful comments and remarks. The author acknowledges financial support from the research project
`Singular perturbation problems for differential operators' Progetto di Ateneo
of the University of Padova. The author is a  member of the Gruppo Nazionale
per l'Analisi Matematica, la Probabilit\`a e le loro Applicazioni (GNAMPA) of
the Istituto Nazionale di Alta Matematica (INdAM).


\begin{thebibliography}{10}

\bibitem{adn}
S. Agmon, A. Douglis, L. Nirenberg, Estimates near the boundary for solutions of elliptic partial differential equations satisfying general boundary conditions. II, Comm. Pure Appl. Math., 17, 35--92, 1964.

\bibitem{buosotesi}
D. Buoso, Ph.D. Thesis, University of Padova, in preparation.

\bibitem{bula2013}
D. Buoso, P.D. Lamberti, Eigenvalues of polyharmonic operators on variable
domains, ESAIM: COCV,
19 (2013), 1225-1235.

\bibitem{buosohinged}
D. Buoso, P.D. Lamberti, Shape deformation for vibrating hinged plates,
Mathematical Methods in the Applied Sciences, 37 (2014), 237-244.

\bibitem{bulareis}
D. Buoso, P.D. Lamberti, Shape sensitivity analysis of the eigenvalues of the Reissner-Mindlin system, to appear on SIAM J. Math. Anal.


\bibitem{bupro}
D. Buoso, L. Provenzano, A few shape optimization results for a biharmonic Steklov problem, in preparation.

\bibitem{henrot}
A. Henrot,  { Extremum problems for eigenvalues of elliptic operators}, Frontiers in Mathematics, Birkh\"{a}user Verlag, Basel, 2006.

\bibitem{lala2004}
P.D. Lamberti, M. Lanza de Cristoforis,  { A real analyticity
result for symmetric functions of the eigenvalues of a domain
dependent Dirichlet problem for the Laplace operator}, J.
Nonlinear Convex Anal 5 (2004), no.1,  19-42.

\bibitem{lala2007}
P.D. Lamberti, M.Lanza de Cristoforis, A real analyticity result for symmetric functions of the eigenvalues of a domain-dependent Neumann problem for the Laplace operator, Mediterr. J. Math. 4 (2007), no. 4, 435-449.
 
\bibitem{lalacri}
P.D. Lamberti, M. Lanza de Cristoforis, {Critical points of the symmetric
functions of the eigenvalues of the Laplace operator and overdetermined problems},
 J. Math. Soc. Japan 58 (2006), no.1,  231-245.

\end{thebibliography}
\end{document}